\documentclass[oneside,english, 11pt]{amsart}
\usepackage[T1]{fontenc}
\usepackage[latin9]{inputenc}
\usepackage{amscd}
\usepackage{amsthm,amsmath,amsfonts,amssymb}
\usepackage{verbatim}
\usepackage{lmodern}
\usepackage{graphicx}
\usepackage{enumerate}

\makeatletter
\numberwithin{equation}{section}
\numberwithin{figure}{section}
\theoremstyle{plain}
\newtheorem{thm}{\protect\theoremname}
  \theoremstyle{plain}
  \newtheorem{lem}[thm]{\protect\lemmaname}

\numberwithin{thm}{section}

\newtheorem{cor}[thm]{Corollary}

\theoremstyle{remark}
\newtheorem*{rem}{Remark}
\makeatother

\usepackage{babel}
\providecommand{\propname}{Proposition}

\providecommand{\lemmaname}{Lemma}
\providecommand{\theoremname}{Theorem}

\newcommand{\imag}{\operatorname{Im} \,}
\newcommand{\real}{\operatorname{Re} \,}
\renewcommand{\Im}{\imag}
\renewcommand{\Re}{\real}

\newcommand{\ee}{\epsilon}

\newcommand{\vp}{\varphi}

\newcommand{\diam}{\operatorname{diam}}

\let \le \leqslant
\let \leq \leqslant
\let \ge \geqslant
\let \geq \geqslant
\let \epsilon \varepsilon
\let \phi \varphi
\let \vp \varphi
\let \ell k

\newcommand{\bbP}{\mathbb{P}}

\newcommand{\bbE}{\mathbb{E}}
\newcommand{\bbN}{\mathbb{N}}

\newcommand{\calA}{\mathcal{A}}
\newcommand{\calK}{\mathcal{K}}

\newcommand{\ga}{\alpha}

\newcommand{\gd}{\delta}
\newcommand{\gD}{\Delta}
\newcommand{\gl}{\lambda}
\newcommand{\gb}{\beta}
\newcommand{\gk}{\kappa}
\newcommand{\gn}{\eta}
\newcommand{\go}{\omega}
\newcommand{\gO}{\Omega}
\newcommand{\gvp}{\varphi}
\newcommand{\gr}{\gamma}
\newcommand{\gs}{\sigma}
\newcommand{\gz}{\zeta}

\newcommand{\wbox}{S_{n; \, j, \ell}}
\newcommand{\corner}{p_{n; \, j, \ell}}
\newcommand{\improve}{\phi(\beta)}

\begin{document}

\title{On the continuity of SLE$_{\kappa}$ in $\kappa$}
\author[F. Johansson Viklund]{Fredrik Johansson Viklund}{
\address{Johansson Viklund: Department of Mathematics\\
Columbia University\\
2990 Broadway, 10027 New York, NY, USA
}
\email{fjv@math.columbia.edu}}
\author[S. Rohde]{Steffen Rohde}
\address{Rohde: Department of Mathematics\\
University of Washington, Seattle\\
Box 354350, Seattle, WA 98195, USA
}
\email{rohde@math.washington.edu}
\author[C. Wong]{Carto Wong}
\address{
 Wong: Department of Mathematics\\
University of Washington, Seattle\\
Box 354350, Seattle, WA 98195, USA
}
\email{carto@u.washington.edu}
\maketitle
\begin{abstract}
We prove that for almost every Brownian motion sample, the corresponding SLE$_\kappa$ curves parameterized by capacity exist and change continuously in the supremum norm when $\kappa$ varies in the interval $[0,\kappa_0)$, where $\kappa_0=8(2-\sqrt{3})\approx 2.1$. We estimate the  $\kappa$-dependent modulus of continuity of the curves and also give an estimate on the modulus of continuity for the supremum norm change with $\kappa$.
\end{abstract}
\section{Introduction and Main Result}
The Schramm-Loewner evolution with parameter $\kappa > 0$, SLE$_\kappa$, is a family of random conformally invariant growth processes that arise in a natural manner as scaling limits of certain discrete models from statistical physics. The construction of SLE uses the Loewner equation, a differential equation that provides a correspondence between a real-valued function --- the Loewner driving term --- and an evolving family of conformal maps called a Loewner chain. If the driving term is sufficiently regular the Loewner chain is generated by (or generates, depending on the point of view) a non self-crossing (continuous) curve which is obtained by tracking the image of the driving term under the evolution of conformal maps. For $\kappa$ fixed and positive, SLE$_\kappa$ is defined by taking a standard one-dimensional Brownian motion $B_t$ and using $\sqrt{\kappa} B_t$ as driving term for the Loewner equation. Despite the fact that there are examples of driving terms strictly more regular than Brownian motion whose corresponding Loewner chains are not generated by (continuous) curves, it is known that for each fixed $\kappa > 0$, the SLE$_\kappa$ Loewner chain almost surely is, see \cite{RS} and \cite{LSW04}. The SLE$_\kappa$ curves are random fractals and as $\kappa$ varies their properties change. For example, when $\kappa$ is between $0$ and $4$ the SLE$_\kappa$ path is almost surely simple, but when $\kappa >4$ it almost surely has double points, and when $\kappa \ge 8$ it is space-filling, see \cite{RS}. The Hausdorff dimension of the curve increases with $\kappa$, see \cite{beffara}, while the H\"older regularity in the standard capacity parameterization derived from the Loewner equation decreases as $\kappa$ increases to $8$ and then the regularity increases again, see \cite{JVL}. In all of these results, the exceptional event can depend on $\kappa$ which is held fixed.   

A natural question that seems to have occurred to several researchers, and is suggested by simulation (see \cite{kennedy}), is whether almost surely the SLE$_\kappa$ curves change continuously with $\kappa$ if the Brownian motion sample is kept fixed. Note that \emph{\`a-priori} it is not even clear that there is an event of full measure on which the corresponding SLE$_\kappa$ Loewner chains are simultaneously generated by curves if $\kappa$ is allowed to vary in an interval. 
An analogous question for the deterministic Loewner equation has been asked by Angel: If the Loewner chain corresponding to the driving term $W_t$ is generated by a continuous curve $\gamma$ and if $\kappa <1$, is it true that the Loewner chain of $\kappa W_t$ is generated by a continuous curve, too? This was answered in the negative by Lind, Marshall, and Rohde by constructing a non-random H\"older-$1/2$ driving term $\lambda_t$ with the property that the Loewner chain of $\kappa \lambda_t$ is generated by a curve if and only if $\kappa \neq \pm 1$, see Theorem~1.2 of \cite{LMR}. More precisely, there exists a special $T > 0$ such that the Loewner chain of $\lambda_t$ is generated by a curve 
$\gamma$ for $t \in [0,T)$ but as $t$ tends to $T$ the curve spirals around a disc in the upper half plane and the limit of $\gamma(t)$ as $t \to T-$ does not exist. The function $\lambda_t$ of this example is strictly more regular than Brownian motion. Indeed, it is well-known that for every $\alpha < 1/2$ the Brownian motion sample path is almost surely H\"older-$\alpha$, but it is almost surely \emph{not} H\"older-1/2. 

In this paper we will prove that SLE$_\kappa$ almost surely does not exhibit the pathological behavior described above, at least not for sufficiently small $\kappa$. Let us state our main result in a slightly informal manner, see Section~\ref{proof-sect} for a precise statement of the full result; we prove more than is stated here. (In particular we will also estimate explicitly the relevant H\"older exponents.) In order to state the theorem, define $\kappa_0 = 8(2-\sqrt{3}) \approx 2.1$.
\begin{thm}\label{main-thm}
For almost every Brownian motion sample $B_t$, the SLE$_\kappa$ Loewner chains driven by $(\sqrt{\kappa} B_t, \, t \in [0,1])$, where $\kappa \in [0, \kappa_0)$, are simultaneously generated by curves that if parameterized by capacity change continuously with $\kappa$ in the supremum norm. 
\end{thm}
See Section~\ref{sect-Whitney} for a sketch of the proof of Theorem~\ref{main-thm} and Theorem~\ref{main-thm2} for a precise statement.

Let us make a few remarks. The restriction to $t \in [0,1]$ is only a convenience and a similar result for $t \in [0,\infty)$ holds if we consider instead continuity with respect to the topology of uniform convergence on compact subintervals. We emphasize that we prove that the curves change continuously with $\kappa$ when the curves have a particular parameterization. This is a stronger topology than the one generated by the now standard metric used by Aizenman and Burchard in \cite{AB} which allows for increasing reparameterization of the curves. It can be checked that Theorem~\ref{main-thm} also holds when $\kappa > \kappa_\infty:= 8(2+\sqrt{3})$ is allowed to vary. This may seem counterintuitive but can be viewed as a consequence of the fact that the regularity of the SLE$_\kappa$ curve in the capacity parameterization increases with $\kappa$ when $\kappa \ge 8$. (Intuitively, by duality, the boundary of the SLE$_\kappa$ hull becomes more and more like the real line when $\kappa$ becomes large and so the H\"older regularity of the curve approaches the minimum of $1/2$ and that of the driving term, which is the time-zero regularity of any chordal Loewner curve in the capacity parameterization.) 

We end with a question. From the point of view of probability theory what we do is to consider a specific coupling of SLE$_\kappa$ processes for different $\kappa$ and prove almost sure existence of the curves and continuity as $\kappa$ varies. As was realized by Schramm and Sheffield \cite{SS} it is possible to obtain SLE$_\kappa$ curves by a mechanism quite different from the usual one using the Loewner equation. Very roughly speaking, the construction considers certain ``flow-lines'' derived from the Gaussian free field (GFF) and by varying a parameter one gets SLE$_\kappa$ for different $\kappa$, see \cite{MS} and the references therein. It seems natural to ask whether a similar continuity result as the one proved in this paper holds for GFF derived couplings of SLE$_\kappa$ for different values of $\kappa$. 
\subsection{Overview of the paper}The organization of our paper is as follows. In Section~\ref{DLE-sect} we discuss the deterministic (reverse-time) Loewner equation and derive Lemma~\ref{L:Df} which estimates the perturbation of a Loewner chain in terms of a small supremum norm perturbation of its driving term. In Section~\ref{SLE-sect} we start by giving the general set up of the proof of the main result along with a sketch its proof. We then give the necessary probabilistic estimates based on previously known moment bounds for the spatial derivative of the SLE map. The complete statement of our main result is given in Theorem~\ref{main-thm2} of Section~\ref{proof-sect}, where the work of Sections \ref{DLE-sect} and \ref{SLE-sect} is then combined to prove Theorems \ref{main-thm} and \ref{main-thm2}. We also prove Theorem~\ref{main-thm3}, a quantitative version of Theorem~\ref{main-thm}.
\subsection*{Acknowledgements}
Fredrik Johansson Viklund acknowledges support from the Simons Foundation, Institut Mittag-Leffler, and the AXA Research Fund, and the hospitality of the Mathematics Department of University of Washington, Seattle. The research of Steffen Rohde and Carto Wong was partially supported by
NSF Grants DMS-0800968 and DMS-1068105.

\section{Deterministic Loewner Equation}\label{DLE-sect}
Let $W_t$ be a real-valued continuous function defined for $t \in [0, \infty)$ and set 
\begin{equation}\label{LPDE}
\partial_{t}f_{t}(z)=-\partial_z f_{t}(z)\frac{2}{z-W_{t}},\quad f_{0}(z)=z.
\end{equation}
This is the (chordal) Loewner partial differential equation and the function $W_t$ is called the Loewner driving term. (As we will only work with the chordal version of the Loewner equation in this paper, we will usually omit the word ``chordal''.) A solution $(f_t(z), \, t\ge 0, \, z \in \mathbb{H})$ exists whenever $W_t$ is measurable and for each $t \ge 0$, $f_t: \mathbb{H} \to H_t$ is a conformal map from the upper half-plane onto a simply connected domain $H_t=\mathbb{H} \setminus K_t$, where $K_t$ is a compact set. We call the family $(f_t)$ of conformal maps a Loewner chain and $(f_t, W_t)$ a Loewner pair. The family of image domains $(H_t)$ is continuously decreasing in the Carath\'eodory sense. We say that the Loewner chain $(f_t)$ is generated by a curve if there is a curve $\gamma(t)$ (that is, a continuous function of $t$ taking values in $\overline{\mathbb{H}}$) with the property that for every $t \ge 0$, $H_t$ is the unbounded connected component of $\mathbb{H}\setminus \gamma[0,t]$. Theorem~4.1 of \cite{RS} gives a convenient sufficient condition for $(f_t)$ to be generated by a curve:
\begin{thm}[\cite{RS}]\label{sufficient}
Let $T>0$ and let $W_t: [0, T] \to \mathbb{R}$ be continuous and $(f_t, W_t)$ the corresponding Loewner pair. Suppose that
\begin{equation}\label{nov18.1}
\beta(t):=\lim_{y \to 0+}f_t(W_t+iy)
\end{equation}
exists for $t \in [0,T]$ and is continuous. Then $(f_t, 0 \le t \le T)$ is generated by the curve $\beta$.
\end{thm}
It is sometimes convenient to write
\[
\hat{f}_t(z)=f_t(W_t+z).
\]
We will usually refer to \eqref{LPDE} as the \emph{Loewner PDE}.

There is another version of the Loewner equation that we shall use, namely the \emph{reverse-time Loewner ODE}
\begin{equation}\label{rev-LODE}
\partial_{t}h(z)=-\frac{2}{h_{t}(z)-W_{t}},\quad h_{0}(z)=z.
\end{equation}
If $(f_t)$ is the solution to the Loewner PDE \eqref{LPDE} with driving term $W_t$ and $(h_t)$ the solution to \eqref{rev-LODE} with driving term $W_{T-t}$, then it is not difficult to see that the conformal maps $f_{T}(z)$ and $h_{T}(z)$ are the same. Note that this identity holds only at
the special time $T$. In particular, the families $(h_{t})$ and
$(f_{t})$ are in general not the same. It is often easier to work with \eqref{rev-LODE} rather than directly with \eqref{LPDE}.

The standard Koebe distortion theorem for conformal maps gives a certain uniform control of the change of a conformal map evaluated at different points at distance comparable to their distance to the boundary. We will need similar estimates to control the change of a Loewner chain evaluated at different times and driven by ``nearby'' driving terms. The magnitude of the allowed perturbation depends on the distance to the boundary of $\mathbb{H}$ and on the behavior of the spatial derivative of the conformal map. We first state the well-known estimates for the $t$-direction, see, e.g., \cite{JVL} for proofs.
\begin{lem}  \label{lemma34} 
    There exists a constant $0 < c < \infty$ such that the following holds. Suppose that $f_t$ satisfies the chordal Loewner PDE \eqref{LPDE} and that $z=x+iy \in \mathbb{H}$. Then for $0 \le s \le y^2$, 
\begin{equation}  \label{nov19.9} 
           c^{-1}  \leq  \frac{\left| f_{t+s}'(z) \right|}{\left| f_t'(z) \right|} \leq   c
\end{equation} 
and  
\begin{equation}  \label{mar14.7}
 \left| f_{t+s}(z) - f_t(z) \right| \leq c \, y \left| f_t'(z) \right| .
\end{equation}
\end{lem} 
The next lemma considers a supremum norm perturbation of the driving term. (One can treat, e.g., the $L^1$ norm with nearly identical arguments.) The most important estimate is \eqref{nov18.4.1} which has appeared in a radial setting in \cite{JV}. It would be sufficient to prove our main result. The refinement \eqref{nov18.4.2} will be used to obtain better quantitative estimates on H\"older exponents using information about the derivative. We stress that we derive \eqref{nov18.4.1} with no assumptions on the driving terms other than the existence of a bound on their supremum norm distance. 
\begin{lem} \label{L:Df}
Let $0 < T < \infty$. Suppose that for $t \in[0 , T]$, $f_t^{(1)}$ and $f_t^{(2)}$ satisfy the chordal Loewner PDE \eqref{LPDE} with $W_t^{(1)}$ and $W_t^{(2)}$, respectively, as driving terms.
Suppose that
\[
\ee=\sup_{s \in [0,T]}\left|W^{(1)}_s - W^{(2)}_s \right|.
\]
Then if $z=x+iy \in \mathbb{H}$,
\begin{equation}\label{nov18.4.1}
\sup_{t \in [0,T]}\left|f_t^{(1)}(z)-f_t^{(2)}(z)\right| \le  \ee \left( \frac{\sqrt{4T+y^2}}{y} -1\right).
\end{equation}
Moreover, for every $t \in [0,T]$, 
\begin{multline}\label{nov18.4.2}
\left|f_t^{(1)}(z)-f_t^{(2)}(z)\right|  \\ \le  \ee  \exp \left\{ \frac{1}{2}  \left[ \log  \frac{I_{t,y} \left|(f_t^{(1)})'(z)\right|}{y}   \log \frac{I_{t,y}\left|(f_t^{(2)})'(z)\right|}{y}  \right]^{1/2} + \log \log \frac{I_{t,y}}{y}\right\},
\end{multline}
where $I_{t,y}=\sqrt{4t+y^2}$.
\end{lem}
\begin{rem}
Since the conformal maps are normalized at infinity there exists a constant $c< \infty$ depending only on $T$ such that for $j=1,2$, $|(f^{(j)}_t)'(z)| \le c(y^{-1}+1)$ for all $z=x+iy \in \mathbb{H}$ and all $t \in [0,T]$. (This is a well-known property of conformal maps but can also be seen from the proof to follow.) Thus if $y \le 1$, say, and $c_1<\infty$ and $\beta_1 \ge -1$ are such that $|(f_t^{(1)})'(z)| \le c_1 y^{-\beta_1}$, then \eqref{nov18.4.2} can be written
\begin{equation}\label{improve-est}
\left|f_t^{(1)}(z)-f_t^{(2)}(z)\right| \le c' \ee y^{-\sqrt{(1+\beta_1)/2}} \log (I_{t,y}y^{-1}),
\end{equation}
where $c' < \infty$ depends only on $T, c_1, \beta_1$.
\end{rem}
\begin{proof}[Proof of Lemma~\ref{L:Df}]
We will start by proving \eqref{nov18.4.2}. Let $t \in (0,T]$ be fixed.
Write 
\[\tilde{W}^{(j)}_s=W^{(j)}_{t-s}, \quad j=1,2.
\]
Let $z=x+iy$ be fixed and set
\[z_s^{(j)}=h^{(j)}_s(z)-\tilde{W}^{(j)}_{s}, \quad j=1,2,
\] where $h^{(j)}$ are assumed to solve \eqref{rev-LODE} with $\tilde{W}^{(j)}, \, j=1,2,$ respectively, as driving terms. 

Define 
\[H(s)=h^{(1)}_s(z)-h^{(2)}_s(z)\] and note that \[H(t) = f_{t}^{(1)}(z) - f_{t}^{(2)}(z).\] Our goal will be to estimate $|H(t)|$. We differentiate $H(s)$ with respect to $s$ and use \eqref{rev-LODE} to obtain a linear differential equation
\[
\dot{H}(s)-H(s)\psi(s)=\left(\tilde W^{(2)}_s-\tilde W^{(1)}_s\right)\psi(s),\]
where  
\[
\psi(s) = \frac{2}{z_s^{(1)}z_s^{(2)}}.
\]
This differential equation can be integrated and with $u(r) = \exp\{-\int_0^r \psi(s) \, ds\}$ we find 
\begin{equation}\label{gron}
H(s)=u(s)^{-1}\left(H(0)+\int_0^s \left(\tilde W^{(2)}_r-\tilde W^{(1)}_r \right)u(r) \psi(r) \, dr \right).
\end{equation}
Hence, upon setting $H(0)=0$,
\[
|H(t)| \le \int_0^{t}\left|\tilde W^{(2)}_s-\tilde W^{(1)}_s\right| e^{\int_s^{t} \Re \psi(r) \, dr}|\psi| \, ds.
\]
Consequently, 
\begin{align}\label{G}
\left|f_{t}^{(1)}(z) - f_{t}^{(2)}(z)\right| & \le \left(\sup_{0 \le s \le t}\left|W^{(2)}_s - W^{(1)}_s \right| \right) \left(\int_0^{t} e^{\int_s^{t} \Re \psi(r) \, dr}|\psi| \, ds   \right)\\ \nonumber
& \le \ee \int_0^{t} e^{\int_s^{t} \Re \psi(r) \, dr}|\psi| \, ds ,
\end{align}
and we see that we need to estimate the last factor in \eqref{G}. We will first prove the bound corresponding to \eqref{nov18.4.2}. Set
\[
x_s+iy_s=z_s^{(1)}, \quad u_s+iv_s=z_s^{(2)}.
\]
Note that \eqref{rev-LODE} implies that for $0 \le s \le t$, 
\begin{equation}\label{nov1211.1}
\log \left(\frac{y_s}{y} \right) = \int_0^s \frac{2}{x_s^2+y_s^2}\, ds, \quad \log | (h_s^{(1)})'(z)|  = \int_0^s \frac{2(x_s^2-y_s^2)}{(x_s^2+y_s^2)^2}\, ds,
\end{equation}
and similarly for $h^{(2)}$. In particular,
\[
\log \left| (f^{(1)}_t)'(z) \right| = \int_0^{t} \frac{2(x_s^2-y_s^2)}{(x_s^2+y_s^2)^2}\, ds,
\]
and similarly for $f_t^{(2)}$.
By the Cauchy-Schwarz inequality we have that
\begin{align*}
\int_s^{t}\Re  \psi(s) \, ds & = \int_s^{t} \frac{2(x_su_s-y_sv_s)}{(x_s^2+y_s^2)(u_s^2+v_s^2)} \, ds \\
& \le \left( \int_0^{t}\frac{2x_s^2}{(x_s^2+y_s^2)^2} \, ds \right)^{1/2} \left( \int_0^{t}\frac{2u_s^2}{(u_s^2+v_s^2)^2} \, ds \right)^{1/2}.
\end{align*}
Here we used that $y_s$ and $v_s$ are always non-negative. We can write
\[
\int_0^{t} \frac{2x_s^2}{(x_s^2+y_s^2)^2} \, ds = \frac{1}{2}\int_0^t \frac{2(x_s^2-y_s^2)}{(x_s^2+y_s^2)^2} \, ds + \frac{1}{2}\int_0^t \frac{2}{x_s^2+y_s^2}\, ds.
\]
It follows from the Loewner equation that $y_t$ and $v_t$ are both bounded above by $(4t + y^2)^{1/2}$, and we conclude using \eqref{nov1211.1} that 
\begin{multline*}
\int_s^{t}\Re  \psi(s) \, ds \\
\le \frac{1}{2}\left(\log|(f_t^{(1)})'(z)| + \log\left(\frac{(4t+y^2)^{1/2}}{y}\right)\right)^{1/2} \\  \times \left(\log| (f_t^{(2)})'(z)| + \log\left(\frac{(4t+y^2)^{1/2}}{y}\right)\right)^{1/2}.
\end{multline*}
We get \eqref{nov18.4.2} by combining the last estimate with \eqref{G} and noting that the Cauchy-Schwarz inequality implies that
\begin{align}\label{nov18.5}
\int_0^t |\psi(s)| \, ds & \le \left(\int_0^t\frac{2}{x_s^2+y_s^2} \, ds \right)^{1/2} \left( \int_0^t\frac{2}{u_s^2+v_s^2} \, ds \right)^{1/2} \\ \nonumber & \le \log\left(\frac{(4t+y^2)^{1/2}}{y} \right).
\end{align}
It remains to prove \eqref{nov18.4.1}. For this, note that 
\[
    \int_0^{t} e^{\int_s^{t} \Re \psi(r) \, dr}|\psi(s)| \, ds \le \int_0^{t} e^{\int_s^{t} |\psi(r)| \, dr}|\psi(s)| \, ds =e^{\int_0^t | \psi(s) | \, ds} -1.
\]
We can then estimate as in \eqref{nov18.5}. Combined with \eqref{G}, this proves \eqref{nov18.4.1} and concludes the proof.
\end{proof}

\section{Schramm-Loewner Evolution and Probabilistic Estimates}\label{SLE-sect}
Let $B_t$ be standard Brownian motion. The Schramm-Loewner evolution SLE$_\kappa$ for $\kappa \ge 0$ fixed is defined by taking $W_t=\sqrt{\kappa}B_t$ as driving term in \eqref{LPDE}. We recall that for each $\kappa \ge 0$, the SLE$_\kappa$ Loewner chain $(f_t^{(\kappa)})$ is almost surely generated by a curve, the (chordal) SLE$_\kappa$ path, $\gr^{(\gk)}$, see \cite{RS} and \cite{LSW04}. We also recall that the tip of the curve at time $t$ is defined by taking the radial limit
\[
\gamma^{(\kappa)}(t):=\lim_{y \to 0+} \hat{f}_t^{(\kappa)}(iy) ,
\]
where  $\hat{f}_t^{(\kappa)}(iy) := f_t^{(\kappa)}(\sqrt{\kappa}B_t+iy)$.
\subsection{Set-up and strategy} \label{sect-Whitney}
\begin{figure}[t]
  \includegraphics[height=60mm]{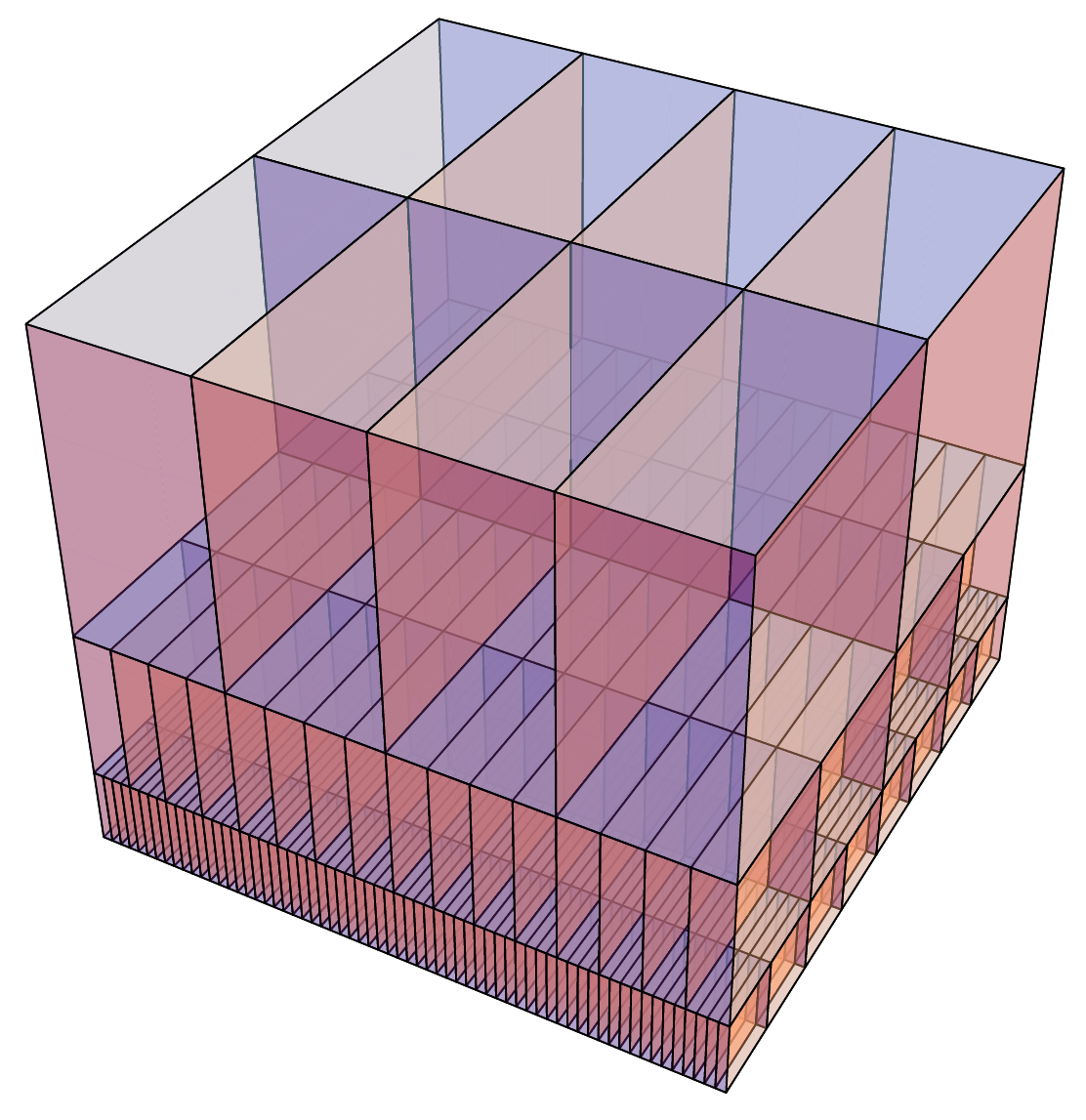}
  \caption{Sketch of the Whitney-type partition of $(t,y,\kappa)$-space. The box volumes decrease with $y$ and the ``floor'' corresponds to $y=0$. The proof controls $\left|(\hat{f}_t^{(\kappa)})'(iy)\right|$ at a corner of each box. } 
  \label{whitfig}
\end{figure}
The idea for the proof of Theorem~\ref{main-thm} is simple and so before giving the details we shall first explain the main steps of the proof and give a few definitions.  
We write 
\begin{equation}\label{june4.1}
F(t,y,\gk) = \hat{f}^{(\gk)}_t(iy), \quad F'(t,y,\gk) = \left(\hat{f}^{(\gk)}_t\right)'(iy),
\end{equation}
and restrict attention to $(t,y, \kappa) \in [0,1] \times [0,1] \times [0, \kappa_0)$.
Our main goal is to show that for almost every Brownian motion sample $B_t=B_t(\omega)$, the function $(t, \kappa) \mapsto \gamma^{(\kappa)}(t)$ defined by taking the radial limit $\lim_{y \to 0+} F(t,y,\kappa)$ is well-defined and continuous for $(t,\kappa) \in [0,1] \times [0, \kappa_0)$. (Recall the sufficient condition of Theorem~\ref{sufficient}.) This will clearly imply Theorem~\ref{main-thm}.
Our strategy is similar to that of the proof of Theorem~5.1 of \cite{RS}. We partition the $(t,y,\kappa)$-space in three-dimensional Whitney-type boxes $\wbox$ whose volumes decrease with the $y$-coordinate: Let
\[
    \wbox = \left[ \frac{j-1}{2^{2n}}, \frac{j}{2^{2n}} \right] \times \left[ \frac{1}{2^n}, \frac{1}{2^{n-1}} \right]
    \times \left[ \frac{\ell - 1}{2^{nq}}, \frac{\ell}{2^{nq}} \right],
\]
where $(n, j, \ell) \in \bbN^3$.  (See Figure~\ref{whitfig} for a sketch.) The parameter $q$ should for now be thought of as being (slightly larger than) $1$. Let
\begin{equation}\label{corner}
    \corner = \left( \frac{j}{2^{2n}}, \frac{1}{2^n}, \frac{\ell}{2^{qn}} \right) \in \wbox
\end{equation}
be the corners of the boxes. The idea is to apply a one-point moment estimate and the Chebyshev inequality to control the magnitude of $|F'|$ at the corners $\corner$ so that for suitable $\gb < 1$ and $j, \ell, n$,
\begin{equation} \label{I:BC}
    \sum_{j,\ell, n} \bbP \left( \left| F'(\corner) \right| \geq 2^{n\gb} \right) \le c \sum_n 2^{qn}2^{2n}2^{-\rho n} < \infty,
\end{equation}
where $\rho$ is the decay rate in the moment estimate we use. (The decay rate of the probabilities in \eqref{I:BC} depend on $\kappa$ and $\beta$ and we need to have $\kappa < \kappa_0$ for the series in \eqref{I:BC} to converge with $\beta < 1$. In particular we need to be able to choose $1< q <\rho-2$.) The Borel-Cantelli lemma then implies that there almost surely exists a random constant $c<\infty$ such that $\left| F'(\corner) \right| \leq c 2^{n\gb}$ for all triples $(n, j, \ell)$ in the sum. With this derivative
estimate we can then use the distortion-type bounds of Section~\ref{DLE-sect} to show that the diameters of the box images decay like a power of the $y$-coordinate, that is,
\begin{equation} \label{I:diamS}
    \diam F(\wbox) \leq c \, 2^{-n\gd},
\end{equation}
where $\gd > 0$ can be thought of as the smaller of $q-1$ and $1-\beta$. Once we have this it is easy to show that $\lim_{y \to 0+} F(t,y,\kappa)$ exists. To prove continuity we estimate 
\begin{align*}
|\gamma^{(\kappa_1)}(t_1) - \gamma^{(\kappa_2)}(t_2)|&=|F(t_1, 0+, \kappa_1) -  F(t_2, 0+, \kappa_2)|\\
& = O\left(|t_1-t_2|^{\delta/2}\right)+O\left(|\kappa_1-\kappa_2|^{\delta/q}\right)
\end{align*} by using \eqref{I:diamS} to sum the diameters of the box images along a ``hyperbolic geodesic'' in $(t,y,\kappa)$-space connecting $(t_1,0,\kappa_1)$ with $(t_2,0,\kappa_2)$. The resulting H\"older exponents depend on the particular choices of parameters ($\beta$ and $q$) and can be taken larger if we restrict attention to smaller $\kappa$. To achieve the best exponents we will prove a local version of \eqref{I:diamS}, which can then be used to patch together a global estimate with varying exponents.  

We now turn to the details.
\subsection{Probabilistic estimates}\label{sect-estimates}
Before stating the basic moment estimate that we will use we need to introduce a few parameters. For $(\gk, \gb) \in (0, \infty) \times (0,1)$, let
\begin{equation} \label{Def:parameters}
    \left\{
    \begin{aligned}
        \gl &= 1 + \frac{2}{\gk} + \frac{\gb(2+\gb) \gk}{8(1+\gb)^2}; \\
        \gz &= \frac{2}{\gk} - \frac{\gb^2 \gk}{8(1+\gb)^2}; \\
        \rho & =\lambda \beta + \zeta.         
    \end{aligned}
    \right.
\end{equation}
It will also be useful to define 
\begin{equation}\label{sigma}
\gs = \min\{ \gl \gb, \rho - 2\}.
\end{equation} 
These notations (with the exception of $\sigma$) with corresponding moment estimates to follow have appeared in, e.g., \cite{JVL} and earlier works by Lawler. We find these estimates more convenient to use and they give better H\"older exponents than those from \cite{RS}, although for technical reasons we shall use a bound from the latter reference when we consider $\kappa$ very close to and including $0$. We remark that Lind \cite{lind-holder} improved the estimates from \cite{RS} in a slightly different setup to essentially agree with the bounds we will use.

\begin{rem}Roughly speaking, the functions in \eqref{Def:parameters} are related in the following way: If $(\kappa, \beta)$ is fixed then as $y \to 0$ the integral $\mathbb{E}[|\hat{f}'_t(iy)|^\lambda]$ is supported on the event that $|\hat{f}'_t(iy)| \approx y^{-\beta}$ and this event has  probability approximately equal to $y^{\lambda \beta + \zeta}=y^\rho$, see \cite{JVL2}. 
\end{rem} 
\begin{thm}[\cite{JVL}] \label{L:momentEst}
    Suppose $(\kappa,\gb) \in (0,\infty) \times (-1,1)$. There is $c = c(\kappa,\gb) <\infty$ such that
    \[
        \bbE \left[ \left| F'\left( \frac{j}{2^{2n}}, \frac{1}{2^n}, \gk \right) \right|^{\gl} \right] \leq c\, j^{-\zeta/2}
    \]
    for all $n \in \bbN$, and $j$ = 1, 2, \dots, $2^{2n}$, where $\gl$ and $\gz$ are defined by \eqref{Def:parameters}.
\end{thm}
Using Theorem~\ref{L:momentEst}, the Chebyshev inequality implies that if $(\kappa, \beta) \in (0, \infty) \times (-1,1)$ then for all $n \in \mathbb{N}$, $j=1,2,\ldots, 2^{2n}$,
\begin{equation}\label{chebyshev}
 \bbP \left\{ \left| F'\left( \frac{j}{2^{2n}}, \frac{1}{2^n}, \gk \right)\right| \geq 2^{n \beta} \right\} \leq c\, 2^{n \lambda \beta} j^{-\zeta/2},
\end{equation}
where $c = c(\gk, \beta) < \infty$. From \eqref{chebyshev}, choosing parameters appropriately, we now get the almost sure control over the derivative at the corners $\corner$ of the boxes $\wbox$ by summing and applying the Borel-Cantelli lemma. Notice that there are $O(2^{n(q+2)})$ boxes at $y$-height $2^{-n}$, where $q$ determines the mesh of the partition in the $\kappa$-direction. The ``optimal'' choice of $q$ depends on which interval of $\kappa$ we consider. It turns out that we need to have $q < \sigma$ for the Borel-Cantelli sums to converge; recall \eqref{I:BC} or see \eqref{A_n} below. On the other hand, the decay rate claimed in
\eqref{I:diamS} becomes
\begin{equation} \label{Def:delta}
    \gd = \min\left\{ q - \improve, 1 - \gb  \right\},
\end{equation}   
where
\[
\improve:=\sqrt{\frac{1+\beta}{2}}
\]
is the exponent from the distortion-type estimate \eqref{improve-est}. Thus we are led to consider $\beta>\hat{\beta}_\kappa$ where $\hat\beta_\kappa$ is a solution in $(0,1)$ to 
\begin{equation}\label{phi-sigma}
\improve=\sigma(\kappa, \beta),
\end{equation}
where $\sigma$ was defined in \eqref{sigma}. If $\beta \ge 0$, then if $\kappa > 1$ we have $\sigma = \rho -2$, while if $\kappa \le 1$, then $\sigma = \lambda \beta$. We have not found a simple expression for $\hat{\beta}_\kappa$ but we note the following properties which can be checked from \eqref{phi-sigma} and \eqref{Def:parameters}. 
\begin{lem}
A solution in $(0,1)$ to the equation \eqref{phi-sigma} exists if and only if $\kappa \in [0,\kappa_0) \cup (\kappa_\infty, \infty)$, where $\kappa_0=8(2-\sqrt{3})$ and $\kappa_\infty=8(2+\sqrt{3})$. For each such $\kappa$, call the solution $\hat{\beta}_\kappa$. Then $\hat{\beta}_\kappa$ increases continuously from $0$ to $1$ as $\kappa$ increases from $0$ to $\kappa_0$ and decreases from $1$ to $0$ as $\kappa$ increases from $\kappa_\infty$ to $\infty$. Moreover, if $\beta > \hat{\beta}_\kappa$ then $\sigma(\kappa, \beta)>\vp(\beta)$. 
\end{lem}
\begin{proof}
We omit the details but note that the special values $\kappa_0, \kappa_\infty$ can be found by solving $\sigma = 1$. 
\end{proof}
\begin{lem}\label{jan18.1}
    Let $\kappa \in [0,\kappa_0)$. If $\beta > \hat{\beta}_\kappa$ and $\vp(\beta) < q < \sigma(\kappa, \beta)$, then there almost surely exists a (random) constant $c = c(\kappa,\gb,q, \omega) < \infty$ such that
    \[
        \left| F'\left( \corner \right) \right| \leq c \, 2^{n \gb}
    \]
    for all $(n, j, \ell) \in \bbN^3$ with $\corner \in [0,1] \times [0,1] \times [0,\kappa]$.
\end{lem}
\begin{proof}The result is clearly true if $\kappa=0$, so let $\kappa \in (0, \kappa_0)$ be fixed and choose $\beta > \hat{\beta}_\kappa$ and
 $\vp(\beta) < q < \sigma(\kappa, \beta)$.
    For $n$ = 1, 2, \dots, let \[\mathcal{A}_n = \mathcal{A}_n(\kappa, \beta, q)\] be the event that $\left| F'(\corner) \right| \geq 2^{n \beta}$ for
   some $(j,\ell) \in \bbN^2$ with $j 2^{-2n} \in [0,1]$ and $\ell 2^{-nq} \in [0,\kappa]$. We have that
  \[
  \mathbb{P}\left( \mathcal{A}_n \right) \le \sum_{j, \ell} \bbP \left\{ \left| F'(\corner) \right| \geq 2^{n \beta} \right\},
  \]
  where the sum is over the above ranges of $j$ and $\ell$. We claim that for all $n=1,2,\ldots$, and $j, \ell$ such that $j 2^{-2n} \in [0,1]$ and $\ell 2^{-nq} \in [0, \kappa]$ we have the uniform estimate
  \begin{equation} \label{I:probEst_1}
      \bbP \left\{ \left| F'(\corner) \right| \geq 2^{n\gb} \right\} \leq c\, j^{-\zeta/2} 2^{-n \gl \gb},
  \end{equation}
  where $c = c(\kappa,\beta) < \infty$, $\zeta=\zeta(\kappa, \beta)$, and $\lambda=\lambda(\kappa, \beta)$. Indeed, this follows from the Chebyshev inequality and Theorem~\ref{L:momentEst} for $j,\ell$ such that $\ell 2^{-nq}$ is contained any fixed closed interval contained in $(0, \kappa]$. For $j,\ell$ such that $\ell 2^{-nq}$ is very close to $0$ we cannot directly quote Theorem~\ref{L:momentEst} since the multiplicative constant in the bound may \emph{\`a-priori} blow up as $\ell 2^{-nq} \to 0$. Moreover, the setup used for the proof of Theorem~\ref{L:momentEst} in \cite{JVL} is such that it would require some work to verify that the constant can be taken to depend only on the largest $\kappa$ considered. Instead, for simplicity and as this is all we need, we will use Corollary~3.5 of \cite{RS}, the proof of which can easily be seen to yield the required uniform constants. The decay rate in Corollary~3.5 of \cite{RS} is not as good as that of Theorem~\ref{L:momentEst} but is still sufficient to imply that \eqref{I:probEst_1} holds with a uniform constant whenever $\ell 2^{-nq}$ is sufficiently small compared to $\kappa$.  We conclude that we may sum (\ref{I:probEst_1}) over $j$ to obtain
   \begin{equation} \label{I:probEst_2}
      \sum_{j=1}^{2^{2n}} \bbP \left\{ \left| F'(\corner) \right| \geq 2^{n \gb} \right\} \leq c \, 2^{-n\gs}
  \end{equation}
  where $\gs=\gs(\kappa, \beta)$ is as in \eqref{sigma} and $c = c(\kappa, \gb) < \infty$. (When performing the summation over $j$ in \eqref{I:probEst_2} we have tacitly, if needed, estimated using a slightly smaller $\beta$ to ensure that $|\zeta/2 - 1|$ is bounded from below.) Summing the last bound over $\ell$ gives
  \begin{equation}\label{A_n}
      \bbP(\calA_n) \leq c\, 2^{-n (\gs - q)}.
  \end{equation}
  The last expression is summable over $n$ and so the proof is complete by the Borel-Cantelli lemma.\end{proof}
We will now apply the uniform derivative estimate of the last lemma to show that the diameters of the $F$-images of the boxes decay like a power of their (minimal, say) $y$-coordinate. Since $\sup_{t \in [0,1]} |\sqrt{\kappa + \Delta \kappa}B_t-\sqrt{\kappa}B_t|$ is of order (a random constant times) $\Delta \kappa$ for $\kappa >0$ but only of order $\sqrt{\Delta \kappa}$ at $\kappa=0$ we must consider these two cases separately.    
    \begin{lem} \label{Cor:diam_S}
    Let $\kappa \in [0,\kappa_0)$. For every $\ee > 0$ there exist $q  > 0$ and $\delta > 0$ and almost surely a (random) constant $c=c(\kappa, q, \ee, \omega) < \infty$ such that 
    \begin{equation}\label{mar12.1}
        \diam F(\wbox) \leq c\, 2^{-n\gd}
    \end{equation}
    for all $(n,j,\ell) \in \bbN^3$ with $\corner \in [0,1] \times [0,1] \times [\ee,\kappa]$. 
    
    Moreover, if $\ee>0$ is sufficiently small there exist $q'  > 0$ and $\delta' > 0$ and almost surely a constant $c=c(q', \ee, \omega) < \infty$ such that \eqref{mar12.1} holds with $\delta$ replaced by $\delta'$ for all $(n,j,\ell) \in \bbN^3$ with $\corner=\corner(q') \in [0,1] \times [0,1] \times [0,\ee]$.

\end{lem}

\begin{proof}
We will start with the first assertion. Let $\ee>0$ be given and assume that $\ee < \kappa$, since there is nothing to prove otherwise.
  Lemma~\ref{jan18.1} shows that if $\beta > \hat{\beta}_\kappa$ and $\vp(\beta) < q < \sigma(\kappa, \beta)$ then there almost surely exists a (random) constant $c = c(\kappa,\gb,q, \omega) < \infty$ such that
  \begin{equation}\label{apr17.1}
      \left| F'\left(p_{n; \, j,\ell} \right) \right| \leq c\, 2^{n \gb}
  \end{equation}
   for all the box corners $\corner  \in [0,1] \times [0,1] \times [\ee,\kappa]=:\mathcal{S}_\ee$. (Note that \eqref{apr17.1} holds also for $\ee=0$.) Consider a fixed but arbitrary dyadic box $\wbox \subset \mathcal{S}_\ee$. Let $p \in \wbox$. We will show that there exists $\delta>0$ such that 
   \begin{equation}\label{apr16.3}
   \left| F(p) - F(\corner) \right| \leq c n 2^{-n\gd}.
   \end{equation} 
   Write
   \[
       (t, y, \gk') = \corner \qquad \text{and} \qquad p = (t + \gD t, y + \gD y, \gk' + \gD \gk) \in \wbox.
   \]
   Since $\left| \gD t \right| \leq y^2$, Lemma~\ref{lemma34} and Lemma~\ref{jan18.1}
  imply that
  \begin{equation} \label{I:DFDt}
      \left| F(t+ \gD t, y, \gk') - F(t, y, \gk') \right| \leq c \, y \left| F'\left(p_{n; \, j,\ell}\right) \right| \leq c' \, 2^{-n(1-\gb)},
  \end{equation}
  where $c' = c'(\kappa,\gb,q, \omega) < \infty$ almost surely. On the other hand,
  \begin{align} \label{I:DFDy}
      \left| F(t+\Delta t, y + \gD y, \gk') - F(t+\Delta t, y, \gk') \right| & \leq c \, y \left| F'\left(p_{n; \, j,\ell} \right) \right| \\ \nonumber & \leq c \, 2^{-n(1-\gb)}
  \end{align}
  by the Koebe distortion theorem and Lemma~\ref{jan18.1} combined with \eqref{nov19.9}. Next, if $\gk'$, $\gk' + \gD \gk \in [\ee, \kappa]$ then
  \begin{equation} \label{I:DW}
      \sup_{t \in [0,1]} \left| \sqrt{\gk' + \gD \gk} B_t - \sqrt{\gk} B_t \right| \leq  c \, \gD \gk \, \sup_{t \in [0,1]}\left| B_t \right|
      \leq c' \, \gD \gk
  \end{equation}
  where $c' = c'(\ee, \omega) < \infty$ almost surely. The estimate~\eqref{improve-est} combined with Lemma~\ref{jan18.1}, Koebe's distortion theorem, and \eqref{nov19.9}, show that
  \begin{align} \label{I:DFDk}
      \left| F(t + \Delta t, y + \Delta y, \gk' + \gD \gk) - F(t + \Delta t, y + \Delta y, \gk') \right| & \leq c \, \gD \gk  \, y^{-\phi(\gb)} \log(y^{-1}) \\ \nonumber
      & \leq c \, n 2^{-n(q - \phi(\gb))}.
  \end{align}
Consequently, by  \eqref{I:DFDt}, \eqref{I:DFDy}, and \eqref{I:DFDk} we get \eqref{apr16.3} with  
\begin{equation}\label{apr17.2}
\delta = \min\{1-\beta, q-\vp(\beta)\},
\end{equation}
which is clearly strictly positive.
It remains to verify the case when $\corner  \in [0,1] \times [0,1] \times [0,\ee]$. For this, note that that all the estimates above except \eqref{I:DW} and \eqref{I:DFDk} hold in this case, too, with the assumption that $\beta' > \hat{\beta}_\ee$ and $\vp(\beta') < q' < \sigma(\ee,\beta')$. We replace \eqref{I:DW} by
  \[
      \sup_{t \in [0,1]} \left| \sqrt{\gk' + \gD \gk} B_t - \sqrt{\gk'} B_t \right| \leq \sqrt{\gD \gk} \sup_{t \in [0,1]} \left| B_t \right|
      \leq c \sqrt{\gD \gk},
  \]
  where $c=c(\omega) < \infty$ almost surely. This gives  
  \begin{equation}\label{apr17.3}
        \left| F(t + \Delta t, y + \Delta y, \gk' + \gD \gk) - F(t +\Delta t, y + \Delta y, \gk') \right| \leq c \, n 2^{-n(q'/2 - \phi(\gb'))}.
  \end{equation}
Note that for $\beta'>0$ fixed $\sigma(\ee, \beta')=O(1/\ee)$ as $\ee \to 0$. Consequently, by taking $\ee>0$ sufficiently small (using also that $\hat{\beta}_\ee$ is increasing in $\ee$) we have that $\beta'>\hat{\beta}_\ee$ and we can find $q'$ with
\[
2 \vp(\beta') < q' < \sigma(\ee, \beta').
\]
By \eqref{apr17.3} we get \eqref{apr16.3} with 
\[
\delta'= \min\{1-\beta', q'/2 - \phi(\gb')\}
\]
strictly positive. This concludes the proof.
\end{proof}

We have seen in the (proofs of the) last two results that the anomalous behavior at $t=0$ and $\kappa=0$ can decrease the decay rate $\delta$ of the box images in \eqref{mar12.1}. In the next lemma we record a quantitative statement which restricts attention to $t,\kappa \ge \ee >0$ and therefore gives better exponents. In this case we can replace the requirement that $\beta > \hat\beta_\kappa$ by $\beta > \beta_\kappa$, where $\beta_\kappa$ is the larger solution to
\[
\varphi(\beta)=\rho(\beta,\kappa)-2.
\] 
Indeed, in this case, the sum \eqref{I:probEst_2} is bounded by an $\ee$-dependent constant times $2^{-n(\rho-2)}$ and not only by $2^{-n\sigma}$ and this implies that we may consider the larger range of $\beta$. (Since $\kappa \ge \ee$ we do not need to use the estimate from \cite{RS}.) We note that $\beta_\kappa$ may be negative. 
\begin{lem} \label{quant-lemma}
    Let $\ee>0$ and let $\kappa \in [\ee,\kappa_0)$. If $\beta > \beta_{\kappa}$ and $\vp(\beta) < q < \rho(\kappa, \beta)-2$ then there almost surely exists a (random) constant $c=c(\ee, \kappa, \beta, q, \omega)<\infty$ such that such that for all $(n,j,\ell) \in \bbN^3$ with $\corner \in [\ee,1] \times [0,1] \times [\ee,\kappa]$,
    \[
        \diam F(\wbox) \leq c \, n 2^{-n\gd},
    \]
     where
    \[
        \gd = \min \left\{1 - \gb, q-\vp(\beta) \right\}.
    \] 
\end{lem}
\section{H\"older Regularity and Proof of Theorem~\ref{main-thm}}\label{proof-sect}
Let us now give a precise statement of Theorem~\ref{main-thm}. Let $(\Omega, \mathbb{P})$ be a probability space supporting a standard linear Brownian motion $B$. For $\omega \in \Omega$ we write $(B_t(\omega), \, t \ge 0)$ for the sample path of $B$. Let $\mathcal{K}$ be the space of continuous curves defined on $[0,1]$ taking values in the closed upper half plane $\overline{\mathbb{H}}=\{z: \Im z \ge 0 \}$. We endow $\mathcal{K}$ with the supremum norm so that it becomes a metric space. 
\begin{thm}\label{main-thm2}
There exists an event $\Omega_* \subset \Omega$ of probability $1$ for which the following holds for every $\go \in \gO_*$. The chordal SLE$_\kappa$ path $(\gr^{(\gk)}(t, \go), \,  t \in [0, 1])$ driven by $(\sqrt{\kappa}B_t(\omega), \, t \in [0,1])$ and parameterized by capacity exists as an element of $\mathcal{K}$ for every $\kappa \in [0, \kappa_0)$, where $\kappa_0=8(2-\sqrt{3})$. Moreover, $\kappa \mapsto \gr^{(\gk)}(\cdot, \go)$ is continuous as a function from $[0, \kappa_0)$ to $(\calK,\Vert \cdot \Vert_{\infty})$.
\end{thm}
\begin{proof}
Let $\kappa \in [0,\kappa_0)$.
We first show that almost surely, for all $(t,\kappa')\in [0,1]\times[0,\kappa]$, $\gr^{(\gk')}(t) = \lim_{y \to 0+} F(t,y,\gk')$ exists; $F$ was defined in \eqref{june4.1}. 
Suppose that $0 < y_1$, $y_2 < 2^{-N}$. The triangle inequality and Lemma~\ref{Cor:diam_S} imply that there is an event $\Omega_\kappa$ of probability $1$ on which there exist a constant $c<\infty$ and $\delta>0$ such that for all $(t, \kappa') \in [0,1]\times[0,\kappa]$,
\[
    \left| F(t, y_1, \gk') - F(t, y_2, \gk') \right| \leq \sum_{n=N}^{\infty} \diam F(Q_n) \leq c \sum_{n=N}^{\infty} 2^{-n \gd},
\]
where for each $n \in \bbN$, $Q_n=Q_n(t,\kappa') \subset \{\wbox\}$ is a Whitney-type box such that $(t, 2^{-n}, \gk') \in Q_n$. (Note that we when we apply Lemma~\ref{Cor:diam_S} we consider separately the two cases when $\kappa'$ is very small and when it is bounded away from $0$ and the Whitney-type partition depends on which of the two cases we apply.) As $N \to \infty$ the right-hand side of the last display converges to $0$ and it follows that $\gamma^{(\kappa)}(t)=\lim_{y \to 0+} F(t,y, \kappa')$ exists for all $(t, \kappa') \in [0,1] \times [0,\kappa]$ on the event $\Omega_\kappa$. Next, we wish to prove that on the event $\Omega_\kappa$, $(t, \kappa) \mapsto \gamma^{(\kappa)}(t)$ is continuous for $(t,\kappa) \in [0,1] \times [0, \kappa]$. For this, let $t_1, t_2 \in [0,1]$ be given. Define the ``stopping time'' $N \in \bbN$ by
\begin{equation} \label{Def:stopping}
    2^{-2N} < \left| t_1 - t_2 \right| \leq 2^{-2(N-1)}.
\end{equation}
(We can assume that $t_1 \neq t_2$.) Note that by the construction of the Whitney-type partition $N=O(-\log |t_1-t_2|^{1/2})$. Using Lemma~\ref{Cor:diam_S} we get, for sufficiently small $y > 0$,
\[
    \begin{aligned}
        \left| F( t_1, y, \gk') - F(t_2, y, \gk') \right|  &\leq \left| F( t_1, y, \gk') - F( t_1, 2^{-N}, \gk') \right|  \\
                                                                                 &\quad + \left| F( t_1, 2^{-N}, \gk') - F( t_2, 2^{-N} , \gk') \right| \\
                                                                              &\quad +\left| F( t_2, 2^{-N}, \gk') - F( t_2,  y , \gk') \right|  \\
        &\leq c \sum_{n=N}^{\infty}  2^{-n \gd} \\
        &\leq c 2^{- N\delta} \\
        &\leq c \left| t_1 - t_2 \right|^{\delta/2}.
  \end{aligned}
\]
We have again tacitly, if needed, considered the two separate cases of Lemma~\ref{Cor:diam_S} and we understand $\delta$ as the smaller of the two exponents obtained. Notice also that the two points $(t_1, 2^{-N}, \kappa')$ and $(t_2, 2^{-N}, \kappa')$ may not be in the same level-$N$ Whitney-type box, so we cannot, strictly speaking, apply Lemma~\ref{Cor:diam_S} to estimate $\left| F( t_1, 2^{-N}, \gk') - F( t_2, 2^{-N} , \gk') \right|$. However, it is clear that the points are contained in a translate of such a box or we can estimate directly as in \eqref{I:DFDt}. A similar remark applies when we estimate $\left| F( t, 2^{-N}, \gk_1) - F( t, 2^{-N} , \gk_2) \right|$ below.  
 
Now, if $\kappa_1, \kappa_2 \in [0,\kappa]$, we use the stopping time $N=O(-\log |\kappa_1-\kappa_2|^{1/q})$ given by
\[
    2^{-qN} < \left| \gk_1 - \gk_2 \right| \leq  2^{- q(N-1)}
\]
instead of \eqref{Def:stopping}. We get
\[
    \begin{aligned}
        \left| F( t, y, \gk_1) - F(t, y, \gk_2) \right|  &\leq \left| F( t, y, \kappa_1) - F( t, 2^{-N}, \gk_1) \right|  \\
                                                                                 &\quad + \left| F( t, 2^{-N}, \gk_1) - F( t, 2^{-N} , \gk_2) \right| \\
                                                                              &\quad +\left| F( t, 2^{-N}, \gk_2) - F( t,  y , \gk_2) \right|  \\
        &\leq c \sum_{n=N}^{\infty}  2^{-n \gd} \\
        &\leq c 2^{- N\delta} \\
        &\leq c \left| \kappa_1 - \kappa_2 \right|^{\delta/q}.
  \end{aligned}
\]
Letting $y \to 0+$ we get that
\[
|\gamma^{(\kappa_1)}(t_1) - \gamma^{(\kappa_2)}(t_2)| \le c|t_1 - t_2|^{\delta/2} + c |\kappa_1 - \kappa_2|^{\delta/q}
\]
holds on the event $\Omega_\kappa$. We may take $\Omega_*=\cap_n\Omega_{\kappa_0-1/n}$ and the proof is complete.

\end{proof}
\subsection{Quantitative Estimates}
We can see from the proof of Theorem~\ref{main-thm2} that we get H\"older continuity in both $t$ and $\kappa$ on any compact subinterval of $[0,\kappa_0)$. We will now state separately a quantitative version of the main result. For simplicity we shall only consider $t, \kappa \ge \ee > 0$, but all cases can be treated with similar arguments.

Recall the definition of $\beta_\kappa$ at the end of Section~\ref{sect-estimates}. We define the following local H\"older exponents. Let
\[
\delta(\beta,q)=\min\{1-\beta, q-\vp(\beta)\}
\]
and then let
\[
    \ga_{\kappa} = \sup_{(\beta, q)} \frac{\delta}{2} \qquad \text{and} \qquad \gn_\kappa = \sup_{(\gb,q)} \frac{\gd}{q},
\]
where the suprema are taken over $\beta>\beta_\kappa$ and $\vp(\beta) < q < \rho(\kappa, \beta)-2$. To find the value of $\ga_{\gk}$, we take $q = \rho(\kappa, \beta) - 2$ and let $\gb$ be the larger solution (for $\beta$) to
\begin{equation} \label{E:alpha}
    1 - \gb = \rho(\kappa, \beta) - 2 - \gvp(\gb).
\end{equation}
We have not found an explicit solution to this equation. However, if we replace $\gvp(\gb)$ by the majorant $\beta/4+3/4$ in \eqref{E:alpha}, then we get a second order polynomial equation in $\beta$ that we can solve to obtain the larger solution $\beta_\kappa'$ below which gives an upper bound for the larger solution to \eqref{E:alpha} and so a lower bound for $\alpha_\kappa$. We have that
\begin{equation}\label{beta-est}
\beta_\kappa'=\frac{-16+\kappa  \left(8+\sqrt{468+30 \kappa }\right)}{16+14 \kappa +\kappa ^2}, \quad \kappa \in [0, \kappa_0),
\end{equation}
which implies the estimate
\[
    \ga_{\gk} \geq \frac{1-\beta'_\kappa}{2} = \frac{1}{2}\left(1 - \frac{-16+\kappa  \left(8+\sqrt{468+30 \kappa }\right)}{16+14 \kappa +\kappa ^2}\right).
\]
As $\gk$ increases from $0$ to $\gk_0 = 8(2-\sqrt{3})$, this lower bound decreases from $1$ to $0$.

In order to estimate $\eta_\kappa$, we fix $\kappa \in [0,\kappa_0)$ and note that for each fixed  $\gb \in (\gb_{\gk},1)$, the function
\[
    q \mapsto \frac{\gd}{q} = \frac{\min\{ 1-\gb, q - \gvp(\gb) \}}{q}
\]
is increasing for $q \in (\gvp(\gb), q_*]$ and is decreasing for $q \in [q_*,\infty)$, where \[q_* = 1 - \gb + \gvp(\gb).\] Recall that we only may take  $q \in (\gvp(\gb), \rho - 2)$ so the maximum occurs
either at $q = q_*$ or $q = \rho- 2$. In fact,
\begin{equation}\label{ca}
    \sup_{q \in (\vp(\beta), \rho - 2)} \frac{\gd}{q} =
    \begin{cases}
        \frac{1-\gb}{1 - \gb + \gvp(\gb)} & \text{ if } q_* \leq \rho - 2 \\
        \frac{\rho - 2 - \gvp(\gb)}{\rho - 2} & \text{ if }  q_* \ge \rho - 2
    \end{cases}
\end{equation}
We now plug in $\beta=\beta'_\kappa$ from \eqref{beta-est} and note that by the definition of $\beta'_\kappa$ it holds that
\[
1-\beta'_\kappa + \vp(\beta'_\kappa) \le 1-\beta'_\kappa + \beta'_\kappa/4+3/4 = \rho(\beta'_\kappa)-2.
\]
Thus with this choice of $\beta$ we see from \eqref{ca} that
\[
\eta_\kappa \ge \frac{1-\gb'_\kappa}{1 - \gb'_\kappa + \gvp(\gb'_\kappa)}.
\]
Again, as $\gk$ increases from $0$ to $\gk_0 = 8(2-\sqrt{3})$, the lower bound decreases from $1$ to $0$.
\begin{thm}\label{main-thm3}
Let $\ee>0$ and $\kappa < \kappa_0$ be given and let $\alpha < \alpha_\kappa$ and $\eta < \eta_\kappa$. There almost surely exists a (random) constant $c=c(\ee, \kappa, \alpha, \eta, \omega)<\infty$ such that for all $(t_j, \kappa_j) \in [\ee,1]\times[\ee,\kappa], \, j=1,2,$
\[
\left| \gamma^{(\kappa_1)}(t_1) - \gamma^{(\kappa_2)}(t_2) \right| \le c |t_1-t_2|^{\alpha} +  c |\kappa_1-\kappa_2|^{\eta}.
\]
\end{thm}
\begin{proof}
This is proved in exactly the same way as Theorem~\ref{main-thm2} but using Lemma~\ref{quant-lemma} instead of Lemma~\ref{Cor:diam_S}.
\end{proof}
By an approximation argument we obtain the following corollary. 
\begin{cor}
Let $\ee, \ee' > 0$ be given. There almost surely exists a (random) constant $c=c(\ee, \ee', \omega) < \infty$ such that for all $t_1,t_2 \in [\ee,1]$ and $\kappa, \kappa_1 \in [\ee,\kappa_0)$ with $\kappa_1 \le \kappa$,
\[
\left| \gamma^{(\kappa_1)}(t_1) - \gamma^{(\kappa)}(t_2) \right| \le c |t_1-t_2|^{\alpha_{\kappa}-\ee'} +  c |\kappa-\kappa_1|^{\eta_{\kappa}-\ee'}.
\]
\end{cor}

\end{document}